\documentclass[a4paper, oneside, 11pt]{amsart}
\usepackage{import}

\title[The balance on representations of a conformal net]{The balanced structure on the category of representations of a conformal net}
\author{Adrià Marín-Salvador}
\date{}

\newcommand{\Z}{\mathbb{Z}}
\newcommand{\R}{\mathbb{R}}
\newcommand{\C}{\mathbb{C}}

\renewcommand{\S}{\mathcal{S}}
\newcommand{\Diff}{\text{Diff}}

\newcommand{\Mob}{\mathbf{M\ddot{o}b}}
\newcommand{\A}{\mathcal{A}}
\newcommand{\Rot}{\text{Rot}}

\newcommand{\id}{\text{id}}

\newcommand{\Hom}{\text{Hom}}

\newcommand{\Rep}{\text{Rep}}

\newcommand{\End}{\text{End}}

\newcommand{\Jcal}{\mathcal{J}}

\newcommand{\B}{\mathbb{B}}
\newcommand{\Vect}{\text{Vect}}

\newcommand{\DHR}{\text{DHR}}

\usepackage{amssymb}
\usepackage{amsfonts}
\usepackage{amsmath}
\usepackage{quiver}
\usepackage{amsthm}
\usepackage{comment}
\usepackage[english]{babel}
\usepackage{quiver}
\usepackage{nicefrac}
\usepackage[a4paper]{geometry}
\usepackage{amssymb}
\usepackage{mathtools}
\usepackage{stackrel}
\usepackage[hidelinks]{hyperref}
\usepackage{mathrsfs} 
\usetikzlibrary{decorations.pathmorphing}
\usetikzlibrary{calc}
\usetikzlibrary{decorations.markings}
\usetikzlibrary{fadings,decorations.pathreplacing}
\usetikzlibrary{matrix,arrows}
\usetikzlibrary{arrows,calc,decorations.pathreplacing,decorations.markings,shapes.geometric,shadows}

\usepackage{xcolor}

\tikzset{-dot-/.style={decoration={
  markings,
  mark=at position 0.5 with {\fill circle (2pt);}},postaction={decorate}}}

\tikzset{
	Fdot/.style={circle, draw, fill, inner sep=0pt}, 
	Odot/.style={circle, draw, inner sep=0.1pt, minimum size=0.1cm}
	}

  \vfuzz \hfuzz
\usepackage{xcolor}
\usetikzlibrary{arrows}

\theoremstyle{definition}
\newtheorem{definition}{Definition}[section]
\newtheorem{proposition}[definition]{Proposition}

\newtheorem{theorem}[definition]{Theorem}
\newtheorem{corollary}[definition]{Corollary}

\newtheorem{remark}[definition]{Remark}

\newtheorem*{theorem*}{Theorem}
\newtheorem*{definition*}{Definition}
\newtheorem*{example*}{Example}
\newtheorem*{corollary*}{Corollary}
\newtheorem*{conjecture*}{Conjecture}

\definecolor{myviolet}{HTML}{7D3C98}
\definecolor{myorange}{HTML}{F39C12}
\definecolor{myblue}{HTML}{2E86C1}
\definecolor{mygreen}{HTML}{1E8449}
\definecolor{myred}{HTML}{C0392B}

\tikzstyle{ipe stylesheet} = [
  ipe import,
  even odd rule,
  line join=round,
  line cap=butt,
  ipe pen normal/.style={line width=0.4},
  ipe pen heavier/.style={line width=0.8},
  ipe pen fat/.style={line width=1.2},
  ipe pen ultrafat/.style={line width=2},
  ipe pen normal,
  ipe mark normal/.style={ipe mark scale=3},
  ipe mark large/.style={ipe mark scale=5},
  ipe mark small/.style={ipe mark scale=2},
  ipe mark tiny/.style={ipe mark scale=1.1},
  ipe mark normal,
  /pgf/arrow keys/.cd,
  ipe arrow normal/.style={scale=7},
  ipe arrow large/.style={scale=10},
  ipe arrow small/.style={scale=5},
  ipe arrow tiny/.style={scale=3},
  ipe arrow normal,
  /tikz/.cd,
  ipe arrows, 
  <->/.tip = ipe normal,
  ipe dash normal/.style={dash pattern=},
  ipe dash dotted/.style={dash pattern=on 1bp off 3bp},
  ipe dash dashed/.style={dash pattern=on 4bp off 4bp},
  ipe dash dash dotted/.style={dash pattern=on 4bp off 2bp on 1bp off 2bp},
  ipe dash dash dot dotted/.style={dash pattern=on 4bp off 2bp on 1bp off 2bp on 1bp off 2bp},
  ipe dash normal,
  ipe node/.append style={font=\normalsize},
  ipe stretch normal/.style={ipe node stretch=1},
  ipe stretch normal,
  ipe opacity 10/.style={opacity=0.1},
  ipe opacity 30/.style={opacity=0.3},
  ipe opacity 50/.style={opacity=0.5},
  ipe opacity 75/.style={opacity=0.75},
  ipe opacity opaque/.style={opacity=1},
  ipe opacity opaque,
]
\definecolor{red}{rgb}{1,0,0}
\definecolor{blue}{rgb}{0,0,1}
\definecolor{green}{rgb}{0,1,0}
\definecolor{yellow}{rgb}{1,1,0}
\definecolor{orange}{rgb}{1,0.647,0}
\definecolor{gold}{rgb}{1,0.843,0}
\definecolor{purple}{rgb}{0.627,0.125,0.941}
\definecolor{gray}{rgb}{0.745,0.745,0.745}
\definecolor{brown}{rgb}{0.647,0.165,0.165}
\definecolor{navy}{rgb}{0,0,0.502}
\definecolor{pink}{rgb}{1,0.753,0.796}
\definecolor{seagreen}{rgb}{0.18,0.545,0.341}
\definecolor{turquoise}{rgb}{0.251,0.878,0.816}
\definecolor{violet}{rgb}{0.933,0.51,0.933}
\definecolor{darkblue}{rgb}{0,0,0.545}
\definecolor{darkcyan}{rgb}{0,0.545,0.545}
\definecolor{darkgray}{rgb}{0.663,0.663,0.663}
\definecolor{darkgreen}{rgb}{0,0.392,0}
\definecolor{darkmagenta}{rgb}{0.545,0,0.545}
\definecolor{darkorange}{rgb}{1,0.549,0}
\definecolor{darkred}{rgb}{0.545,0,0}
\definecolor{lightblue}{rgb}{0.678,0.847,0.902}
\definecolor{lightcyan}{rgb}{0.878,1,1}
\definecolor{lightgray}{rgb}{0.827,0.827,0.827}
\definecolor{lightgreen}{rgb}{0.565,0.933,0.565}
\definecolor{lightyellow}{rgb}{1,1,0.878}
\definecolor{black}{rgb}{0,0,0}
\definecolor{white}{rgb}{1,1,1}

\usetikzlibrary{arrows.meta,patterns}
\usetikzlibrary{ipe}

\begin{document}

\begin{abstract}
    Let $\A$ be a (not necessarily rational) conformal net. We show that the braided $\mathrm{W}^*$-tensor category $\Rep(\A)$ of representations of $\A$ is canonically a balanced $\mathrm{W}^*$-tensor category. The balance is given by the action of $e^{-2\pi i L_0}$, where $L_0$ denotes the generator of rotations on $S^1$. In \cite{GcrossedbraidedRep}, we generalize this result to the larger context of a group acting on $\A$. We provide here a more accessible proof for the case where no group is~present.
\end{abstract}

\maketitle

\section{Introduction}

Conformal nets provide a formalization of 2-dimensional unitary chiral conformal field theory \cite{bmt88, bsm90, bgl93, gf93, w95, kl04, bdh15}. A \emph{conformal net} consists of a vacuum Hilbert space $H_0$ equipped with a unitary projective action of the group $\Diff^+(S^1)$ of orientation-preserving diffeomorphisms of the circle together with an assignment of a von Neumann algebra $\A(I)\subset B(H_0)$ to every interval $I\subset S^1$ of the circle. The assignment $I\mapsto \A(I)$ is compatible with the inclusion of intervals and with the $\Diff^+(S^1)$ action.

Given a conformal net $\A = (\A, H_0)$, a \emph{representation} of $\A$ is a Hilbert space $H$ equipped with $*$-actions of all the von Neumann algebras $\A(I)$ compatible with the inclusions $\A(J)\subset \A(I)$ for all pairs of intervals $J\subset I$. We denote by $\Rep(\A)$ the category whose objects are $\A$-representations and whose morphisms are morphisms of Hilbert spaces intertwining the actions of all the algebras $\A(I)$. Given a representation $H\in \Rep(\A)$, the algebra $\End_{\Rep(\A)}(H)$ of endomorphisms of $H$ is a von Neumann algebra, making $\Rep(\A)$ a $\mathrm{W}^*$-category in the sense of \cite{henriques2024completewcategories}. In addition, $\Rep(\A)$ is canonically a braided $\mathrm{W}^*$-tensor category. The tensor structure of $\Rep(\A)$ can be constructed in different equivalent ways, see \cite[Sec. IV.2.]{gf93}, \cite[Sec. 30]{was98}, \cite[Sec. 1.C]{bdh15} and \cite[Sec. 2]{Gui21}. For the braiding, see \cite[Sec. 2]{frs}, \cite[Sec. 7]{lon89}, \cite[Def. 4.16]{gf93} and the modern treatment of \cite[Sec. 2]{Gui21}. Recall that, given a braided $\mathrm{W}^*$-tensor category $\mathcal{C}$, with braiding $\beta$, a \emph{balance} (or \emph{categorical twist}) on $\mathcal{C}$ consists of a natural family of unitary isomorphisms $\theta_X:X\to X$ such that
\[
\theta_{X\otimes Y} = (\theta_X\otimes \theta_Y)\circ\beta_{Y,X}\circ \beta_{X,Y}
\]
for all $X,Y\in \mathcal{C}$. We prove that $\Rep(\A)$ is canonically a balanced $\mathrm{W}^*$-tensor category, see Theorem \ref{Thm: RepAutAIsCrossedBalanced}.

\begin{theorem*}
    Let $\A$ be a (not necessarily rational) conformal net. The braided $\mathrm{W}^*$-tensor category $\Rep(\A)$ of representations of $\A$ is canonically a balanced $\mathrm{W}^*$-tensor category. 
\end{theorem*}

The balance is constructed as follows. Given a representation $H\in \Rep(\A)$, it was proved in \cite{MR2078164} (see also \cite{Hen19}) that $H$ carries an action of the group $\widetilde{\Mob}$, the universal cover of the group $\Mob$ of Möbius transformations on the circle. Denoting by $\exp: \R\to S^1$ the universal cover of $S^1$ and by $\Diff(\R)$ the group of diffeomorphisms of $\R$, we have $\widetilde{\Mob} = \{(\tilde{f}, f)\in \Diff(\R)\times \Mob\, |\, f\circ \exp = \exp\circ \tilde{f}\}$. Then, the balance $\theta_H$ on $H$ is given by the action of $(x\mapsto x-1, \id)\in \widetilde{\Mob}$, that is, the action of $e^{-2\pi i L_0}$, for $L_0$ the generator of rotations on $S^1$.

A large part of the literature focuses on Möbius covariant nets, where the vacuum $H_0$ is only required to carry an action of $\Mob$. Then, representations of the net do not carry a canonical action of $\widetilde{\Mob}$ and hence do not carry a canonical balance. Similarly, even when $\Diff^+(S^1)$-covariant nets are studied, they are mostly treated in the language of DHR (for Doplicher, Haag, Roberts) endomorphisms \cite{MR297259, MR334742, gf93}. In that language, the balance does not have a natural definition in general. In the particular case where $\A$ is rational, then $\Rep(\A)$ is a unitary rigid tensor category (from results in \cite{MR1838752, MR2100058}), and therefore it does have a canonical balance. The conformal spin and statistics Theorem from \cite{MR1410566} implies that our balance agrees with the canonical one in the case where $\A$ is rational, see Theorem \ref{Thm: SpinStatistics} and the discussion leading~to~it.

In this note, we rely on the construction of the fusion and the braiding on $\Rep(\A)$ defined in \cite{Gui21}. In Section \ref{Sec: ConfNets} we recall the definition of conformal nets and their representations, in Section \ref{Sec: CatOfReps} we recall the braided tensor structure on $\Rep(\A)$ of \cite{Gui21}, and in Section \ref{sec: balance} we construct the balance and prove the necessary compatibility with the braiding.

\section*{Acknowledgements}
This work has been funded by the EPSRC grant EP/W524311/1.  

For the purpose of Open Access, the author has applied a CC BY public copyright license to any Author Accepted Manuscript version arising from this submission.

\addtocontents{toc}{\protect\setcounter{tocdepth}{5}}

\section{Conformal nets}
\label{Sec: ConfNets}

Let $S^1:=\{z\in \C\ |\ |z| = 1\}$ denote the standard circle. We define an \emph{interval} of $S^1$ to be an open, connected, non-empty, non-dense subset of $S^1$, and we denote by $\Jcal$ the collection of intervals of $S^1$. Let us denote by $\Mob$ the set of Möbius transformations of $S^1$, that is, transformations of the form
\[
z\mapsto \frac{az+b}{\bar{b}z+\bar{a}}
\]
with $a,b\in \C$, $|a|^2-|b|^2 = 1$. Then, $\Mob$ is a Lie group isomorphic to $PSL(2, \R).$ In particular, the group $\Rot$ of rotations of the circle embeds in $\Mob$. We denote an anticlockwise rotation of angle $t$ by $R_t$. In this paper, all Hilbert spaces and $\mathrm{C}^*$-algebras are assumed to be separable.

\begin{definition}
    A \textit{Möbius covariant net on $S^1$} is a tuple $(H_0, \A, U, \Omega)$ where $H_0$ is a Hilbert space equipped with a non-zero vector $\Omega\in H_0$, $U$ is a strongly continuous unitary representation of $\Mob$ on $H_0$ and $\A$ is an assignment of a von Neumann algebra $\A(I)$ acting on $H_0$ for every interval $I\in\Jcal$. This data is required to satisfy, for every $I,\, J\in\Jcal$ and $\varphi\in\Mob$,
    \begin{enumerate}
        \item isotony: if $J\subset I$, then $\A(J)\subset \A(I)$;
        \item locality: if $I\cap J = \emptyset$, then $\A(J)$ and $\A(I)$ commute in $B(H_0)$;
        \item Möbius covariance: $U(\varphi)\A(I)U(\varphi)^* = \A(\varphi I)$;
        \item positivity of the energy: the representation $U$ is of positive energy, meaning that the conformal Hamiltonian $L_0$, defined by $U(R_t) = e^{itL_0}$ is positive;
        \item uniqueness of the vacuum vector: $\Omega\in H_0$ is the unique vector, up to a phase, which is invariant under $U$;
        \item cyclicity of the vacuum: $\Omega$ is cyclic for the von Neumann  algebra $\A(S^1):=\bigvee\limits_{I\in \Jcal}\A(I)\subset B(H_0)$.
    \end{enumerate}
\end{definition}

Let $\Diff^+(S^1)$ denote the group of orientation-preserving diffeomorphisms of $S^1$. A strongly continuous projective unitary representation $V$ of $\Diff^+(S^1)$ on a Hilbert space $H$ is a strongly continuous homomorphism $V: \Diff^+(S^1)\to \mathcal{P}U(H)$. Note that $\Mob\subset\Diff^+(S^1)$. We say that $V$ is an extension of a unitary representation $U$ of $\Mob$ on $H$ if, for every $\varphi\in \Mob$, it holds that $V(\varphi) = [U(\varphi)]$, where $[-]: U(H)\to \mathcal{P}U(H)$ denotes the projection map.

\begin{definition}
    A Möbius covariant net $(H_0, \A ,U, \Omega)$ is said to be a \emph{conformal net} if there is an extension of $U$ to a unitary projective representation of $\Diff^+(S^1)$, which we also denote by $U$, such that for all $I\in\Jcal$ and $\varphi\in \Diff^+(S^1)$, it holds that
    \begin{enumerate}
        \item $U(\varphi)\A(I) U(\varphi)^* = \A(\varphi I)$;
        \item if $\varphi|_{I} = \id_I$, then $\text{Ad}(U(\varphi))|_{\A(I)} = \id_{\A(I)}.$
    \end{enumerate}
\end{definition}

The following are well-known properties of any conformal net $(H_0, \A, U, \Omega)$. 
\begin{enumerate}
    \item Haag duality: the commutant of $\A(I)$ in $B(H_0)$ is $\A(I^c)$, where $I^c$ denotes the interior of $S^1\setminus I$ \cite[Thm. 2.19]{gf93};
    \item The Reeh-Schlieder Theorem: the vacuum vector $\Omega$ is cyclic and separating for every von Neumann algebra $\A(I)$ \cite[Cor. 2.8]{gf93};
    \item Each von Neumann algebra $\A(I)$ is a type $\mathrm{III}_1$ factor \cite[Prop. 1.2]{MR1410566}.
\end{enumerate}

Let $(H_0,\A,U, \Omega)$ be a conformal net, which we will simply denote by $\A$ from now on. Let us introduce the category of representations of $\A$.

\begin{definition}\label{def: Rep}
    A \emph{representation} of a conformal net $\A$ consists of a Hilbert space $H$ and a collection of $*$-homomorphisms $\pi_I:\A(I)\to B(H)$ for every interval $I\in\Jcal$ such that
\[
\pi_I|_{\A(J)} = \pi_J
\]
whenever $J\in\Jcal$ is an interval such that $J\subset I$. We write $\Rep(\A)$ for the category whose objects are representations of $\A$ and whose morphisms are bounded linear intertwining the actions of all the algebras $\A(I)$.
\end{definition}

We note that if the underling Hilbert space of a representation $(H, \pi)$ is separable, which we always assume in this paper, the $*$-actions $\pi_I$ are automatically normal. The conformal net $\A$ has a canonical representation on its vacuum Hilbert space $H_0$ by definition, which we call the \emph{vacuum representation of} $\A$. Given $I\in\Jcal$ and $x\in \A(I)$, we write $\pi_{0,I}(x)$ or $\pi_0(x)$ for the action of $x$ on the vacuum representation $H_0$.

\section{The braided category of representations of a conformal net}
\label{Sec: CatOfReps}
The category $\Rep(\A)$ is a braided $\mathrm{W}^*$-tensor category. The tensor structure of $\Rep(\A)$ can be constructed in different equivalent ways, see \cite[Sec. IV.2.]{gf93}, \cite[Sec. 30]{was98}, \cite[Sec. 1.C]{bdh15} and \cite[Sec. 2]{Gui21}. For the braiding, see \cite[Sec. 2]{frs}, \cite[Sec. 7]{lon89}, \cite[Def. 4.16]{gf93} and the modern treatment of \cite[Sec. 2]{Gui21}. We recall here the construction of \cite{Gui21}. Actually, we will produce the opposite of the braided monoidal category of Gui, in order to agree with our conventions in the work \cite{GcrossedbraidedRep} (see Remark \ref{Rk: Opposite}).

\subsection{Connes fusion of representations}\label{Sec: ConnesFusion}

Let $\A$ be a conformal net. Given $I\in \Jcal$, we write $I^{c}\in \Jcal$ for the interior of $S^1\setminus I$, and call it the \emph{complement} of $I$. Given non-zero representations $H,\, K\in \Rep(\A)$, and $I\in \Jcal$, we write $\Hom_{\A(I)}(H, K)$ for the vector space of bounded operators $T:H\to K$ such that, for all $x\in \A(I)$, it holds that $T\circ \pi^H_{I}(x) = \pi_I^K(x)\circ T$. Since $\A(I)$ is a type $\mathrm{III}$ factor, there is an equivalence $H\cong K$ of $\A(I)$-modules. Hence, there are unitary operators in $\Hom_{\A(I)}(H, K)$. Fix a representation $H = (H, \pi^H)\in \Rep(\A)$ and let $I\in \Jcal$.

\begin{definition}
    A vector $\xi\in H$ is said to be $I$\emph{-bounded} if there exists $T\in \Hom_{\A(I^{c})}(H_0, H)$ such that $\xi = T\Omega$. We write $H(I)\subset H$ for the subspace of $I$-bounded vectors.
\end{definition}

Given $\xi\in H(I)$, there exists a unique operator $T\in \Hom_{\A(I^{c})}(H_0, H)$ such that $\xi = T\Omega$, as $\A(I^{c})\Omega$ is dense in $H_0$ by the Reeh-Schlieder Theorem. We denote such operator by $T = Z(\xi, I)$. In particular, for the vacuum representation we have that $H_0(I) = \A(I)\Omega$, by Haag duality. This implies that $H_0(I)\subset H_0$ is dense for all $I\in\Jcal$. Since there exist unitary operators in $\Hom_{\A(I^{c})}(H_0, H)$, we obtain that $H(I)\subset H$ are dense inclusions for all $I\in\Jcal$. In addition, if $I_1\subset I$ is an inclusion of intervals in $\Jcal$, then $H(I_1)\subset H(I)$ is also dense. 

Let $K = (K, \pi^K)\in \Rep(\A)$ be another representation and $I,\, J\in \Jcal$ be disjoint intervals. We define a positive sesquilinear form $\langle - |-\rangle$ on $H(I)\otimes K(J)$ by setting
\[
\langle \xi\otimes\eta|\xi'\otimes\eta'\rangle:=\langle Z(\eta', J)^*Z(\eta, J)Z(\xi', I)^*Z(\xi, I)\Omega|\Omega\rangle
\]
 for $\xi, \xi'\in H(I)$ and $\eta, \eta'\in K(J)$. Note that $Z(\xi', I)^*Z(\xi, I): H_0\to H_0$ is in $\Hom_{\A(I^{c})}(H_0,H_0)\cong \A(I)$, and $Z(\eta', J)^*Z(\eta, J)\in \Hom_{\A(J^c)}(H_0, H_0) \cong \A(J)$. Hence, the maps $Z(\xi', I)^*Z(\xi, I)$ and $Z(\eta', J)^*Z(\eta, J)$ commute by locality, and the inner product above also equals
\[
\langle \xi\otimes\eta|\xi'\otimes\eta'\rangle=\langle Z(\xi', I)^*Z(\xi, I)Z(\eta', J)^*Z(\eta, J)\Omega|\Omega\rangle.
\]
The positivity of $\langle-|-\rangle$ follows from \cite[Prop. IX.3.15]{MR1943006}.

\begin{definition}\label{def: ConnesFusion}
    The Hilbert space $H(I)\boxtimes K(J)$ is the completion of $H(I)\otimes K(J)$ with respect to the inner product $\langle -|-\rangle$. We call it the \emph{Connes fusion of $H$ and $K $ over the intervals~$I, J.$} 
\end{definition}

We continue denoting by $\xi\otimes \eta$ the image of a vector $\xi\otimes \eta\in H(I)\otimes K(J)$ in $H(I)\boxtimes K(J)$. Recall that we can regard $Z(\xi', I)^*Z(\xi, I)$ as an element of $\A(I)$, which we denote by the same symbol. We have that
\begin{align}\label{eq: OneSidedInnerProd}
    \langle \xi\otimes\eta|\xi'\otimes \eta'\rangle  = \langle \pi^K_{I}(Z(\xi', I)^*Z(\xi, I))\eta|\eta'\rangle  
  = \langle \pi^H_{J}(Z(\eta', J)^*Z(\eta, J))\xi|\xi'\rangle.
\end{align}
Given $I_1,J_1\in\Jcal$ disjoint such that $I_1\subset I$, $J_1\subset J$, there is a canonical equivalence
\begin{equation}\label{eq: CanonicalEquivalenceInclusion}
H(I_1)\boxtimes K(J_1)\xrightarrow{\cong} H(I)\boxtimes K(J)
\end{equation}
induced by the inclusion $H(I_1)\otimes K(J_1)\hookrightarrow H(I)\otimes K(J)$. We also have a canonical equivalence
\begin{equation}
\label{eq: Swap}
H(I)\boxtimes K(J)\xrightarrow{\cong} K(J)\boxtimes H(I)
\end{equation}
induced by the map $H(I)\otimes K(J)\xrightarrow{} K(J)\otimes H(I)$ defined as $\xi\otimes\eta\to \eta\otimes \xi$.

Let $\hat{H},\,\hat{K}\in \Rep(\A)$ be two representations, and $F: H\to \hat{H}$ and $G: K\to \hat{K}$ morphisms in $\Rep(\A)$. We denote by $F\boxtimes G : H(I)\boxtimes K(J)\to \hat{H}(I)\boxtimes \hat{K}(J)$ the bounded map induced~by 
\begin{equation}\label{eq: FusionOfMaps}
    (F\boxtimes G)(\xi\otimes\eta) = F(\xi)\otimes G(\eta)
    \end{equation}
    for all $\xi\in H(I)$ and $\eta\in K(J)$.

We can also define the Connes fusion over pairs of points. Given $z, \zeta\in S^1$ distinct, let
\[
H(z)\boxtimes K(\zeta):=\varinjlim\limits_{(z,\zeta)\in I\times J}H(I)\boxtimes K(J) = \Bigg(\bigsqcup\limits_{(z,\zeta)\in I\times J}H(I)\boxtimes K(J)\Bigg)\Big/\cong,
\]
where the limit runs over disjoint intervals $I,\, J\in \Jcal$ with $z\in I$ and $\zeta\in J$, and $\cong$ denotes the equivalence obtained by identifying $H(I_1)\boxtimes K(J_1)$ with $H(I)\boxtimes K(J)$ by the unitary \eqref{eq: CanonicalEquivalenceInclusion} whenever $I_1\subset I$ and $J_1\subset J$. Given $I,\, J\in \Jcal$ disjoint with $z\in I$ and $\zeta\in J$, the canonical~map 
\[
H(I)\boxtimes K(J)\to H(z)\boxtimes K(\zeta)
\]
is a unitary equivalence. 

We will next relate the Connes fusions over different pairs of intervals. Given intervals $I_i,\, J_i\in \Jcal$ with $J_i\subset I_i^c$ for $i = 0,1$, and a suitable path $\gamma$ in $S^1\times S^1$ from a point in $I_0\times J_0$ to a point in $I_1\times J_1$, we wish to define a unitary $\gamma^\bullet: H(I_0)\boxtimes K(J_0)\xrightarrow{\cong} H(I_1)\boxtimes K(J_1)$. Let $\text{Conf}_2(S^1):=\{(z,\zeta)\in S^1\times S^1\,|\, z \neq \zeta\}$ be the space of configurations of two points in $S^1$. Let $\gamma = (\alpha, \beta):[0,1]\to \text{Conf}_2(S^1)$ be a path from $\gamma(0) =: (z_0, \zeta_0)$ to $\gamma(1) =: (z_1, \zeta_1)$. Assume that $\gamma$ is small enough so that $\gamma([0,1])\subset I\times J$, for some disjoint intervals $I,\, J\in \Jcal$. We define the unitary equivalence $\gamma^\bullet: H(z_0)\boxtimes K(\zeta_0)\xrightarrow{\cong}H(z_1)\boxtimes K(\zeta_1)$ as the composition
\[
H(z_0)\boxtimes K(\zeta_0)\xrightarrow{\cong}H(I)\boxtimes K(J)\xrightarrow{\cong}H(z_1)\boxtimes K(\zeta_1).
\]
For a general path $\gamma$, we choose $0 = t_0<t_1<\ldots <t_k = 1$ such that each $\gamma_n:=\gamma|_{[t_n, t_{n+1}]}$ is small enough in the sense above, and define $\gamma^\bullet: H(z_0)\boxtimes K(\zeta_0)\xrightarrow{\cong}H(z_1)\boxtimes K(\zeta_1)$ as 
\[
\gamma^\bullet = \gamma_{k-1}^\bullet\ldots\gamma_1^\bullet\gamma_0^\bullet.
\]
Since finer partitions give the same result, the map $\gamma^\bullet$ is independent of the choice of partition. We call $\gamma^\bullet$ the \emph{path-continuation} induced by $\gamma$.

Let $I_0, I_1, J_0, J_1\in \Jcal$ be intervals such that $J_0\subset I_0^{c}$ and $J_1\subset I_1^{c}$. Let $\gamma$ be a path in $\text{Conf}_2(S^1)$ from $\gamma(0) =: (z_0, \zeta_0)\in I_0\times J_0$ to $\gamma(1) =: (z_1, \zeta_1)\in I_1\times J_1$. We define the path continuation $\gamma^\bullet : H(I_0)\boxtimes K(J_0)\xrightarrow{\cong} H(I_1)\boxtimes K(J_1)$ as the composition
\[
H(I_0)\boxtimes K(J_0)\xrightarrow{\cong}H(z_0)\boxtimes K(\zeta_0)\xrightarrow{\gamma^\bullet} H(z_1)\boxtimes K(\zeta_1)\xrightarrow{\cong}H(I_1)\boxtimes K(J_1).
\]
Homotopic paths induce the same path-continuation, see \cite[Prop. 2.11]{Gui21}.

In Definition \ref{def: ConnesFusion}, we have defined the Connes fusion of $H $ and $K $ over a pair of disjoint intervals $I,\, J\in \Jcal$ as a certain completion of the vector space $H(I)\otimes K(J)$. The same Hilbert space can be obtained as a completion of the vector spaces $H(I)\otimes K$ and $H\otimes K(J)$. We define $H(I)\boxtimes K$ to be the completion of $H(I)\otimes K$ under the inner product
\[
\langle \xi\otimes\eta\,|\,\xi'\otimes \eta'\rangle := \langle\pi^K_{I}\big(Z(\xi', I)^*Z(\xi, I)\big)(\eta)\,|\, \eta'\rangle
\]
for $\xi, \xi'\in H(I)$ and $\eta, \eta'\in K $, and call it the \emph{Connes fusion of $H $ and $K $ over $I$ on the left}. Then, if $J\in \Jcal$ is an interval disjoint from $I$, the inclusion $H(I)\otimes K(J)\hookrightarrow H(I)\otimes K $ induces a canonical equivalence $H(I)\boxtimes K(J)\xrightarrow{\cong} H(I)\boxtimes K $, by Equation \eqref{eq: OneSidedInnerProd}. For any $z\in S^1$, we define $H(z)\boxtimes K  := \varinjlim\limits_{z\in I}\, H(I)\boxtimes K $. Given a path $\alpha$ in $S^1$ from $z_0$ to $z_1$, we also obtain a path continuation $\alpha^\bullet: H(z_0)\boxtimes K \xrightarrow{\cong} H(z_1)\boxtimes K $ as follows. If $\alpha$ is small enough so that $\alpha([0,1])$ can be covered by an interval $I\in \Jcal$, we define
\[
\alpha^\bullet: H(z_0)\boxtimes K \xrightarrow{\cong} H(I)\boxtimes K \xrightarrow{\cong} H(z_1)\boxtimes K .
\]
For a generic path $\alpha$, we pick a partition $0 = t_0<t_1<\ldots<t_k = 1$ of $[0,1]$ so that each $\alpha_n:=\alpha|_{[t_n, t_{n+1}]}$ is small enough in the sense above, and define $\alpha^\bullet = \alpha_{k-1}^\bullet\ldots\alpha_0^\bullet$. Given two intervals $I_0,I_1\in \Jcal$ and a path $\alpha$ in $S^1$ from a point in $I_0$ to a point in $I_1$, we define
\[
\alpha^\bullet:H(I_0)\boxtimes K \xrightarrow{\cong } H(\alpha(0))\boxtimes K \xrightarrow{\alpha^\bullet} H(\alpha(1))\boxtimes K \xrightarrow{\cong} H(I_1)\boxtimes K .
\]
Homotopic paths induce the same unitary equivalence. Path continuations for Connes fusions over one and two intervals are related by the following result.

\begin{proposition}(\cite[Prop. 2.12]{Gui21})\label{prop: SinglePathContinuationWRTDoublePathContinuation}
    Let $I_0,I_1,J_0,J_1\in \Jcal$ be such that $J_i\subset I_i^{c}$ for $i = 0,1$. Let $\gamma = (\alpha, \beta)$ be a path in $\text{Conf}_2(S^1)$ from $I_0\times J_0$ to $I_1\times J_1$. Then, the path-continuation $\alpha^\bullet: H(I_0)\boxtimes K \xrightarrow{\cong} H(I_1)\boxtimes K $ is given by
    \[
    H(I_0)\boxtimes K \xrightarrow{\cong} H(I_0)\boxtimes K(J_0)\xrightarrow{\gamma^\bullet} H(I_1)\boxtimes K(J_1)\xrightarrow{\cong} H(I_1)\boxtimes K .
    \]
\end{proposition}
Similar properties hold for  $H \boxtimes K(J)$, the \emph{Connes fusion of $H $ and $K $ over $J$ on the right}, defined as the completion of $H \otimes K(J)$ with respect~to
\[
\langle \xi\otimes\eta\,|\,\xi'\otimes\eta'\rangle:=\langle \pi^H_{J}\big(Z(\eta', J)^*Z(\eta, J)\big)(\xi)\,|\,\xi'\rangle
\]
for $\xi, \xi'\in H $ and $\eta,\eta'\in K(J)$. Note that, as in the Connes fusion over two intervals, we have canonical equivalences
\(
H(I)\boxtimes K\xrightarrow{\cong} K\boxtimes H(I). 
\)

\subsection{Endowing the Connes fusion with an action of $\A$}

From now on, we fix disjoint intervals $I,\, J\in \Jcal$, as well as $\A$-representations $H  = (H , \pi^H)$ and $K  = (K , \pi^K)$. We will endow the Hilbert space $H (I)\boxtimes K (J)$ with an action $\pi^{H\boxtimes K}$ of $\A$. Note that we have natural actions of $\A(I )$ and $\A(J)$ on $H (I)\boxtimes K (J)$: given $x\in \A(I)$ and $y\in \A(J)$, we can define
\[
\pi^{H \boxtimes K}_{I}(x)(\xi\otimes \eta) = \Big(\pi_{I}^H(x)(\xi)\Big)\otimes \eta\hspace{2cm}
\pi^{H\boxtimes K}_{J }(y)(\xi\otimes\eta) = \xi\otimes \Big(\pi_{J}^K(y)(\eta)\Big)
\]
for $\xi\otimes\eta \in H(I)\otimes K(J)\subset H(I)\boxtimes K(J) $. This also provides canonical actions of all $\A(L )$ for $L \subset I$ or $L \subset J$.

\begin{theorem}(\cite[Thm. 2.15]{Gui21})\label{thm: RepresentationOnFusion}
    Let $H $ and $K $ be representations of a conformal net $\A$. Let $I,\, J\in\Jcal$ be disjoint intervals. Then,
    \begin{enumerate}
        \item there is a unique action $\pi^l$ of $\A$ on $H(I)\boxtimes K(J)$ such that the following condition holds: for every disjoint intervals $L ,M\in \Jcal$ and any path $\gamma$ in $\text{Conf}_2(S^1)$ from $I\times J$ to $L \times M $, it holds that, for all $x\in \A(L)$ \begin{equation}\label{eq: leftRep}
        \pi_{L}^l(x) = (\gamma^\bullet)^{-1} (\pi_{L }^{H}(x)\boxtimes \id)\gamma^\bullet;
        \end{equation}
        \item  there is a unique action $\pi^r$ of $\A$ on $H(I)\boxtimes K(J)$ such that the following condition holds: for every disjoint $L ,M \in \Jcal$ and any path $\vartheta$ in $\text{Conf}_2(S^1)$ from $I\times J$ to $M \times L$, it holds that, for all $x\in \A(L )$, \begin{equation}\label{eq: rightRep}
        \pi^r_{L }(x) = (\vartheta^\bullet)^{-1} (\id \boxtimes \pi_{L }^{K}(x))\vartheta^\bullet;
        \end{equation}
        \item the actions $\pi^l$ and $\pi^r$ agree, and hence there is a natural action $\pi^{H\boxtimes K} = \pi^l = \pi^r$ of $\A$ on $H(I)\boxtimes K(J)$;
    \item the unitary maps induced by inclusions of intervals and path continuations between Connes fusions intertwine the actions $\pi^{H\boxtimes K}$.
    \end{enumerate}
\end{theorem}

Let $\hat{H},\, \hat{K}\in \Rep(\A)$ be representations of $\A$. Recall that, given morphisms of $\A$-representations $F:H \to \hat{H}$ and $G:K \to \hat{K}$ we have defined a bounded map (see Equation \eqref{eq: FusionOfMaps})
\[
F\boxtimes G: H(I)\boxtimes K(J)\to \hat{H}(I)\boxtimes \hat{K}(J).
\]
By \cite[Prop. 2.17]{Gui21}, $F\boxtimes G$ commutes with canonical equivalences given by inclusion and restriction of intervals, and path continuations, and is a morphism of $\A$-representations.

We can define similarly actions of $\A$ on fusions over single intervals. Let $I\in \Jcal$. If $L \in \Jcal$ is a subinterval $L \subset I$, we define the action of $x\in\A(L )$ on $H(I)\boxtimes K $ by $\pi_{L }^{H\boxtimes K}(\xi\otimes \eta) = \pi_{L }^H(x)(\xi)\otimes \eta$ for all $\xi\in H(I)$ and $\eta\in K $. For a generic interval $L \in\Jcal$, we choose a path $\alpha$ in $S^1$ from $I$ to $L $, and define $\pi_{L }^{ H\boxtimes K}(x) = (\alpha^\bullet)^{-1}\circ \big(\pi_{L }^H(x)\boxtimes\id\big)\circ\alpha^\bullet$. This is independent of the path chosen. Alternatively, one can define the action of $\A$ on $H(I)\boxtimes K $ via the canonical equivalence $H(I)\boxtimes K(J)\cong H(I)\boxtimes K $ for some $J\in \Jcal$ with $J\subset I^{c}$ and the action of $\A$ on $H(I)\boxtimes K(J)$ from Theorem \ref{thm: RepresentationOnFusion}. These two actions agree.

Given a third representation $R\in \Rep(\A)$ and disjoint intervals $I,\, J\in\Jcal$, we can consider the representation $(H(I)\boxtimes K)\boxtimes R(J)$, obtained as the Connes fusion of $H(I)\boxtimes K$ and $R$ over $J$ on the right. Similarly, we can consider $H(I)\boxtimes (K\boxtimes R(J))$. By \cite[Sec. 2.5]{Gui21}, the map $(H(I)\otimes K)\otimes R(J)\to H(I)\otimes (K\otimes R(J))$ given by $(\xi\otimes \eta)\otimes \chi\to \xi\otimes(\eta\otimes \chi)$ induces a unitary equivalence of $\A$-representations
\begin{equation}\label{eq: Associator}
(H(I)\boxtimes K)\boxtimes R(J)\xrightarrow{\cong} H(I)\boxtimes (K\boxtimes R(J)).
\end{equation}

\subsection{Conformal structures}

The fact that any representation of a $\A$ is conformal covariant will provide the balance in Section \ref{sec: balance}. We recall here how to produce this conformal structure.

Let $q: \R\to S^1$ be the map $q(t) = e^{2\pi i t}$ and write $\widetilde{\Diff^+}(S^1) = \{(\widetilde{f}, f)\in \Diff(\R)\times \Diff^+(S^1)\,|\,q\circ \tilde{f} = f\circ q\}$ for the universal cover of $\Diff^+(S^1).$ We can lift the projective representation $U$ of $\Diff^+(S^1)$ on $H_0$ to a projective representation of $\widetilde{\Diff^+}(S^1)$ on $H_0$, which we continue denoting by $U$, given by $U(\widetilde{f},f) = U(f)$. Denoting the image of a unitary $V\in U(H_0)$ in $\mathcal{P}U(H_0)$ by $[V]$, we define the topological group
\[
\Diff_\A^+(S^1) = \{(\widetilde{f},f,V)\in \widetilde{\Diff^+}(S^1)\times U(H_0)\,|\,[V]= U(\widetilde{f},f)\}.
\]
The group $\Diff_\A^+(S^1)$ inherits a topology from that of $\widetilde{\Diff^+}(S^1)\times U(H_0)$, and is a central extension of $\widetilde{\Diff^+}(S^1)$ by $U(1)\cong \{V\in U(H_0)\,|\, [V] = [\id]\}$. A unitary action of $\Diff_\A^+(S^1)$ on a Hilbert space $H$ is, by definition, a continuous homomorphism $\Diff_\A^+(S^1)\to U(H)$ such that the central $U(1)$ acts in the standard way. There is a canonical unitary action of $\Diff_\A^+(S^1)$ on $H_0$ given by $(\widetilde{f},f,V)\in \Diff_\A^+(S^1)\mapsto V\in U(H_0)$. We continue denoting this action~by~$U$. 

Given an interval $I\in \Jcal$, we denote by $\Diff_I(S^1)$ the subgroup of $\Diff^+(S^1)$ of diffeomorphisms with support in $I$. We denote by $\widetilde{\Diff_I}(S^1)$ the connected branch of the preimage of $\Diff_I(S^1)$ in $\widetilde{\Diff^+}(S^1)$ containing the identity, and by $\Diff^+_\A(I)$ the preimage of $\widetilde{\Diff_I}(S^1)$ in $\Diff_\A^+(S^1).$ By conformal covariance of $\A$, given $(\widetilde{f},f,V)\in \Diff_\A^+(I)$, we have $U(\widetilde{f},f,V)\in \Hom_{\A(I^c)}(H_0, H_0)\cong \A(I)$. We will write $U(\widetilde{f},f,V)\in \A(I)$. These elements form the Virasoro subnet inside $\A$.

\begin{theorem}(\cite[Thm. 2.2]{Gui21})\label{Thm: RepIsConformal}
    Let $H \in \Rep(\A)$ be a representation of $\A$. Then, $H$ is conformal covariant in the sense that there exists a unique unitary action $U^H$ of $\Diff^+_\A(S^1)$ on $H $ such that, for all $I\in \Jcal$ and $(\tilde{f},f,V)\in \Diff_\A^+(I)$, it holds that
    \[
    U^H(\widetilde{f},f,V) = \pi^H_I(U(\widetilde{f},f,V)).
    \]
\end{theorem}

\begin{corollary}(\cite[Cor. 2.6]{Gui21})\label{cor: RepOfDiffA}
    For any $(\widetilde{f},f,V)\in \Diff^+_\A(S^1)$, it holds that $U^H(\widetilde{f},f,V)\in \bigvee\limits_{I\in \Jcal}\pi^H_I(\A(I))$, and for every $I\in \Jcal$ and $x\in \A(I)$, it holds that
    \[
    U^H(\widetilde{f},f,V)\circ \pi_{I }^H(x)\circ U^H(\widetilde{f},f,V)^* = \pi_{fI }^H\big(U(\widetilde{f},f,V)\, x \,U(\widetilde{f},f,V)^*\big). 
    \]
\end{corollary}

The Lie algebra of $\Diff^+(S^1)$ is the Lie algebra $\Vect(S^1)$ of vector fields on $S^1$ with the negative of the Lie bracket of vector fields. We let $\exp: \Vect(S^1)\to \Diff^+(S^1)$ be the exponential map and write $\Vect_{\C}(S^1)$ for the complexification of $\Vect(S^1)$. For $n\in \Z$, we define the complex vector field $L_n(e^{i\theta}):=-ie^{in\theta}\frac{d}{d\theta}\in \Vect_{\C}(S^1)$. These vectors $L_n$ form the Witt algebra $\mathscr{W}$, which is a dense Lie subalgebra of $\Vect_{\C}(S^1).$ We define a $*$-structure on $\mathscr{W}$ by setting $L_n^* = L_{-n}$. Then, a vector $X\in \mathscr{W}$ is self-adjoint if and only if $iX\in \Vect(S^1)$. For a self-adjoint vector $X\in \Vect_\C(S^1)$, we write $\exp_{iX}$ for the one-parameter subgroup of $\Diff^+(S^1)$ given by $t\in\R\mapsto \exp(itX)$, so that $\exp_{iL_0}$ is the subgroup of rotations of $S^1$. We also define $\widetilde{\exp}_X: \R\to \widetilde{\Diff^+}(S^1)$ for the one-parameter subgroup of $\widetilde{\Diff^+}(S^1)$ lifting $\exp_X$, and $\widetilde{\exp}(X):=\widetilde{\exp}_X(1)$.

Note that $\widetilde{\Diff^+}(S^1)$ contains the universal covering space $\widetilde{\Mob}$ of $\Mob$, which is generated by $\widetilde{\exp}(iX)$ where $X = \overline{a_{1}}L_{-1} + a_0L_0 + a_1 L_1\in\Vect_{\C}(S^1)$ with $a_0\in\R$ and $a_1\in\C$. The unitary projective representation $U^H$ of $\widetilde{\Diff^+}(S^1)$ on $H $ restricts to a unitary projective representation of $\widetilde{\Mob}$. By \cite{MR58601}, such a projective representation lifts to a unique unitary representation of $\widetilde{\Mob}$ on $H $. Given $X$ as above, we write $e^{iX}$ for the action of $\widetilde{\exp}(iX)\in \widetilde{\Mob}$ on $H $.

For the rest of this section, we will describe the conformal structure of the Connes fusion $H(I)\boxtimes K$ of two $\A$-representations $H$ and $K$ over an interval $I\in \Jcal$. The arguments are taken from \cite[Sec. 2.4]{Gui21}. Let $(\tilde{f},f,V)\in \Diff_A^+(S^1)$ and choose a map $\lambda:[0,1]\to \Diff_A^+(S^1)$ from $1 = (\id,\id,1)$ to $(\tilde{f},f,V).$ We require that $\lambda$ descends to a continuous path $[\lambda]:[0,1]\to \widetilde{\Diff^+}(S^1)$. The homotopy class of $[\lambda]$ is uniquely determined by $(\tilde{f},f,V)$. Given $I\in\Jcal$ and $z\in I$, the map
\[
\begin{array}{cccc}
    \lambda_z: &[0,1]&\to &S^1  \\
     & t&\mapsto & \text{proj}_{\Diff(\S^1)}(\lambda(t))(z)
\end{array}
\]
is a path from $I$ to $fI$, where $\text{proj}_{\Diff(S^1)}:\Diff_A^+(S^1)\to\widetilde{\Diff^+}(S^1) \to \Diff(S^1)$ is the projection.

We define the unitary map $U_0^{H\boxtimes K}(\tilde{f},f,V):H(I)\boxtimes K \xrightarrow{\cong} H(I)\boxtimes K $ as the completion~of
\begin{align}\label{eq: ConformalStructOnConnesFusion0}
U_0^{H\boxtimes K}&(\tilde{f},f,V)(\xi\otimes\eta) \\&=\nonumber (\lambda_z^\bullet)^{-1}\big(U^H(\tilde{f},f,V)\circ Z(\xi,I)\circ U(\tilde{f},f,V)^{-1}(\Omega)\otimes U^K(\tilde{f},f,V)(\eta)\big)
\end{align}
for any $\xi\otimes \eta\in H(I)\otimes K $. This unitary is well-defined by \cite[Sec. 2.4]{Gui21}.

\begin{theorem}(\cite[Thm. 2.21]{Gui21})\label{Thm: ConformalStructureOfFusion}
    Let $H,\, K \in \Rep(\A)$ be representations of $\A.$ Then, for any $I\in\Jcal$, the unitary representation $U^{H\boxtimes K}$ of $\Diff^+_\A(S^1)$ defining the conformal structure of $H(I)\boxtimes K $ is given by the action $U^{H\boxtimes K}_0$.
\end{theorem}

\subsection{The braided $\mathrm{W}^*$-structure}

In this section, we recall how to equip the category $\Rep(\A)$ with the structure of a braided $\mathrm{W}^*$-tensor category following \cite[Sec. 2.6]{Gui21}.

Let $S_-^1$ and $S_+^1$ denote the lower and upper hemispheres of $S^1$ respectively. Given $\A$-representations $H,\, K \in \Rep(\A)$, we define their tensor product to be $$H \boxtimes K :=H(S^1_-)\boxtimes K(S^1_+).$$ We also identify $H \boxtimes K $ with $H(S^1_-)\boxtimes K $ and $H \boxtimes K(S^1_+)$ using the canonical equivalences. We extend this tensor product linearly to obtain a functor
\[
\boxtimes: \Rep(\A)\times \Rep(\A)\to \Rep(\A).
\]

\begin{remark}\label{Rk: Opposite}
    In \cite{Gui21}, the author defines the Connes fusion to be $H\boxtimes K = H(S^1_+)\boxtimes K(S^1_-)$. We define the monoidal opposite category to Gui. Similarly, we will produce the reverse braiding to his. We do this so that our results agree with those in \cite{GcrossedbraidedRep}. There, we extend the results of Gui to the setting where there is a group $G$ acting on the conformal net $\A$. This produces a $G$-crossed braided tensor category, as opposed to a braided tensor category. While in the braided context it is a choice to take the braided category of Gui or its opposite, that choice disappears in the $G$-crossed context. Restricting the $G$-crossed case to $G = \{e\}$ the trivial group, we obtain the choice we make here, as opposed to the one in \cite{Gui21}.
\end{remark}

Let us denote $S^1_-$ and $S^1_+$ by $-$ and $+$ in the Connes fusions. Given $H ,K ,R \in\Rep(\A)$ three representations of $\A$, we define the associator $(H \boxtimes K )\boxtimes R \xrightarrow{\cong}H \boxtimes (K \boxtimes R )$ to~be
\[
(H \boxtimes K )\boxtimes R \cong (H(-)\boxtimes K )\boxtimes R(+)\xrightarrow{\cong } H(-)\boxtimes (K \boxtimes R(+)) {\cong}   H \boxtimes (K \boxtimes R ),
\]
where we use the unitary equivalence \eqref{eq: Associator}. The pentagon axioms can be easily verified, see \cite[Fig. 2.3]{Gui21}. The unit object is $H_0$ and the unitors are described in \cite[p.26]{Gui21}. 

We now introduce the braiding on $\Rep(\A)$. Let $\varrho:[0,1]\to S^1$ be the path $\varrho(t) = e^{\pi i (t-\nicefrac{1}{2})}$, which goes from $S_-^1$ to $S_+^1$. Given representations $H,\, K \in\Rep(\A)$, we define the braid operator $\B_{H,K}: H \boxtimes K \xrightarrow{\cong} K \boxtimes H $~by
\[
\B_{H,
K}: H \boxtimes K  = H(-)\boxtimes K \xrightarrow{\varrho^\bullet} H(+)\boxtimes K \xrightarrow{\cong} K \boxtimes H (+) = K \boxtimes H,
\]
using the equivalence \eqref{eq: Swap}. The braid operators are clearly natural in $H $ and $K $, and they satisfy the hexagon equations by \cite[Thm. 3.8]{Gui21}. Hence, we obtain the following result.

\begin{theorem}(\cite[Thm. 3.9]{Gui21})
     Let $\A$ be a conformal net. The $\mathrm{W}^*$-category $\Rep(\A)$ of representations of $\A$, equipped with the Connes fusion of representations and the braiding $\B$, is a braided $\mathrm{W}^*$-tensor category.
\end{theorem}

\section{The balance}
\label{sec: balance}

We now provide the braided $\mathrm{W}^*$-tensor category $\Rep(\A)$ with a balance. On a representation $H \in \Rep(\A)$, we define the balance as follows.

\begin{definition}\label{def: Balance}
    Let $H \in \Rep(\A)$ be a representation of $\A$. We denote by
    \[
    \theta_H:=e^{ -2\pi i L_0}: H \to H 
    \]
    the unitary on $H $ given by the action of $\widetilde{\exp}(-2\pi iL_0)\in \widetilde{\Mob}$.
\end{definition}

\begin{proposition}
    The unitary $\theta_H = e^{-2\pi i L_0}: H \xrightarrow{\cong} H 
    $ is an equivalence 
    of $\A$-representations.
\end{proposition}
\begin{proof}
    Let $I \in \Jcal$ and $x\in \A(I )$. Then, by Corollary \ref{cor: RepOfDiffA}, we have 
    \[
    e^{-2\pi iL_0}\circ \pi^H_{I}(x)\circ e^{2\pi i L_0} = \pi^H_{I}(e^{-2\pi i L_0 }\, x\, e^{2\pi i L_0}) = \pi^{H}_{I}(x),
    \]
    where we use that the action of $e^{-2\pi i L_0}$ on $\A(I)$ is the identity.
\end{proof}

In order to show that $\theta$ provides a balance on $\Rep(\A)$, we need to understand the action $e^{-2\pi i L_0}$ on the Connes fusion $H(I)\boxtimes K $ of two representations $H$ and $K$ over a given interval $I\in \Jcal$. We do so by specializing Theorem \ref{Thm: ConformalStructureOfFusion}. 

\begin{corollary}\label{cor: e2piiL0OnConnesFusion}
    Let $z\in I$ be arbitrary and $\lambda_z$ be the loop in $S^1$ based at $z$ given by $t\in[0,1]\mapsto e^{-2\pi i t}z$. Then, for any $\xi\otimes \eta \in H(I)\otimes K\subset H(I)\boxtimes K$, it holds that
\[
e^{-2\pi i L_0}(\xi\otimes \eta) = (\lambda_z^{\bullet})^{-1}(e^{-2\pi i L_0}\xi\otimes e^{-2\pi i L_0}\eta).
\]
\end{corollary}
\begin{proof}
    The claim follows from Theorem \ref{Thm: ConformalStructureOfFusion} and the fact that, if $(\tilde{f},f,V)\in \Diff_\A^+(S^1)$ is in the preimage of $\widetilde{\Mob}$, then $U^H(\tilde{f},f,V)\circ Z(\xi,I)\circ U(\tilde{f},f,V)^{-1}(\Omega) = U^H(\tilde{f},f,V)\,\xi$, since $\Omega$ is fixed by $\Mob$.
\end{proof}

We are now ready to prove the main theorem of the paper.

\begin{theorem}\label{Thm: RepAutAIsCrossedBalanced}
    Let $\A$ be a (not necessarily rational) conformal net. The braided $\mathrm{W}^*$-category $\Rep(\A)$ of representations of $\A$, equipped with the family $\theta$, is a balanced $\mathrm{W}^*$-tensor category.
\end{theorem}
\begin{proof}
    The family $\theta$ is natural by Theorem \ref{Thm: RepIsConformal}. We have to show that, for every $H,\, K\in\Rep(\A)$, the following diagram commutes
\begin{equation}\label{eq: BalanceCondition}\begin{tikzcd}
	{H\boxtimes K} && {H\boxtimes K} \\
	{K\boxtimes H} && {H\boxtimes K}.
	\arrow["{\theta_{H\boxtimes K}}", from=1-1, to=1-3]
	\arrow["{\mathbb{B}_{H,K}}"', from=1-1, to=2-1]
	\arrow["{\mathbb{B}_{K,H}}"', from=2-1, to=2-3]
	\arrow["{\theta_H\boxtimes\theta_K}"', from=2-3, to=1-3]
\end{tikzcd}\end{equation}
     Let $H,\, K\in \Rep(\A)$. We first compute the composition $\mathbb{B}_{K,H}\circ \mathbb{B}_{H,K}$. We write $e^{\pi i }\cdot \varrho:[0,1]\to S^1$ for the path $t\mapsto e^{\pi i (t + \nicefrac{1}{2})}$, and $\tau:[0,1]\to S^1$ for $t\mapsto e^{\pi i (2t-\nicefrac{1}{2})}$. It follows that $\tau = (e^{\pi i }\cdot \varrho)\ast \varrho$. In addition, the pair $(\varrho, e^{\pi i  }\cdot \varrho)$ is a path in $\text{Conf}_2(S^1)$ from $S^1_-\times S^1_
     +$ to $S^1_+\times S^1_-$. The composition $\mathbb{B}_{K,H}\circ \mathbb{B}_{H,K}$ is given, by definition, by the top-right leg of the following diagram.
\[\begin{tikzcd}
	{H(-)\boxtimes K} & {H(+)\boxtimes K} & {K\boxtimes H(+)} & {K(-)\boxtimes H(+)} \\
	{H(-)\boxtimes K} &&& {K(-)\boxtimes H} \\
	{K\boxtimes H(-)} &&& {K(+)\boxtimes H} \\
	{K(+)\boxtimes H(-)} & {K(+)\boxtimes H} && {H\boxtimes K(+)}.
	\arrow["{\varrho^\bullet}", from=1-1, to=1-2]
	\arrow["{\tau^\bullet}"', from=1-1, to=2-1]
	\arrow["\cong", from=1-2, to=1-3]
	\arrow["{(e^{\pi i }\cdot\varrho)^\bullet}", from=1-2, to=2-1]
	\arrow["\cong", from=1-3, to=1-4]
	\arrow["{(e^{\pi i }\cdot\varrho)^\bullet}", from=1-3, to=3-1]
	\arrow["\cong", from=1-4, to=2-4]
	\arrow["{(\varrho, e^{\pi i }\cdot\varrho)^\bullet}", from=1-4, to=4-1]
	\arrow["\cong"', from=2-1, to=3-1]
	\arrow["{\varrho^\bullet}", from=2-4, to=3-4]
	\arrow["{\varrho^\bullet}", from=2-4, to=4-2]
	\arrow["\cong"', from=3-1, to=4-1]
	\arrow["\cong", from=3-4, to=4-4]
	\arrow["\cong"', from=4-1, to=4-2]
	\arrow["\cong"', from=4-2, to=4-4]
\end{tikzcd}\]
We claim that the outer diagram commutes, which we show by arguing the commutativity of all the inner diagrams. We discuss the commutativity of the inner diagrams from top to bottom. The first triangle commutes by compatibility of path continuations with concatenation of paths. The next quadrilateral commutes trivially, and the two below commute by Proposition \ref{prop: SinglePathContinuationWRTDoublePathContinuation}. The bottom diagram also commutes trivially. Hence, the commutativity of diagram \eqref{eq: BalanceCondition} is equivalent to the commutativity of the outer diagram in
\[\begin{tikzcd}
	{H(-)\boxtimes K} &&&&&& {H(-)\boxtimes K} \\
	\\
	& {H(-)\boxtimes K} &&&&& {H\boxtimes K(+)} \\
	\\
	{H(-)\boxtimes K} &&&&&& {H\boxtimes K(+)}.
	\arrow["{e^{- 2\pi i L_0}}", from=1-1, to=1-7]
	\arrow["{e^{-2\pi i L_0}\boxtimes e^{-2\pi i L_0}}", from=1-1, to=3-2]
	\arrow["{\tau^\bullet}"', from=1-1, to=5-1]
	\arrow["\cong", from=1-7, to=3-7]
	\arrow["{\tau^\bullet}", from=3-2, to=1-7]
	\arrow["{e^{-2\pi i L_0}\boxtimes e^{-2\pi i L_0}}"', from=5-1, to=1-7]
	\arrow["\cong", from=5-1, to=5-7]
	\arrow["{e^{-2\pi i L_0}\boxtimes e^{-2\pi i L_0}}"', from=5-7, to=3-7]
\end{tikzcd}\]
The upper triangle commutes by Corollary \ref{cor: e2piiL0OnConnesFusion} and the fact that $(\lambda_z^\bullet)^{-1}$ can be taken to be the path-continuation $\tau^\bullet$. The middle diagram commutes by the fact that the $\Diff_\A^+(S^1)$-action commutes with the maps induced by inclusions of intervals, and hence with path-continuations. The bottom diagram commutes trivially. Therefore, $\theta$ provides a balance on $\Rep(\A)$.
\end{proof}

In the case when the conformal net $\A$ is rational, the category $\Rep(\A)$ is a unitary fusion category (from results in \cite{MR1838752, MR2100058}). Then, every object $H\in \Rep(\A)$ comes with a dual representation $\overline{H}\in \Rep(\A)$ and evaluation and coevaluation morphisms
\[
\text{ev}_{H}: \overline{H}\boxtimes H \to H_0\hspace{2cm} \text{coev}_{H}: H_0\to H\boxtimes \overline{H}
\]
satisfying the snake identities. By unitarity, $\Rep(\A)$ also comes equipped with a preferred pivotal structure, that is a family of compatible isomorphisms $\pi_H: H\to \overline{\overline{H}}$ indexed by $H\in \Rep(\A)$. We can define a balance $\theta'$ on $\Rep(\A)$ by the usual picture of a kink, meaning
\begin{align*}
\theta'_H: H\cong H\boxtimes H_0&\xrightarrow{\id\boxtimes\text{coev}_{H}}H\boxtimes H\boxtimes\overline{H}\\&\xrightarrow{\mathbb{B}_{H,H}\boxtimes \id} H\boxtimes H\boxtimes\overline{H} \xrightarrow{\id\boxtimes\pi_{H}\boxtimes \id}H\boxtimes\overline{\overline{H}}\boxtimes\overline{H} \xrightarrow{\id\boxtimes \text{ev}_{\overline{H}}} H\boxtimes H_0\cong H.
\end{align*}
It is therefore natural to ask if, in this case, the balance $\theta'$ agrees with the balance we have defined in Theorem \ref{Thm: RepAutAIsCrossedBalanced}. This is essentially the conformal spin and statistics Theorem in \cite{MR1410566}. There, the authors consider the category $\DHR(\A)$ of DHR endomorphisms of $\A$ instead of $\Rep(\A)$, but these categories are well-known to be equivalent. See \cite{MR297259, MR334742, frs, MR1199171} for an introduction to DHR endomorphisms, and \cite[Sec. 6 and Thm. 6.6]{Gui21} for the equivalence of braided tensor categories.

Actually, to obtain the braiding $\mathbb{B}$ from Section \ref{Sec: CatOfReps}, we need to take the opposite of the monoidal structure and the braiding of $\DHR(\A)$ considered in \cite{MR1410566}. Indeed, by the way endomorphisms are localized in the definition of the braiding in \cite[p. 8]{MR1410566}, it is clear their conventions agree with those in \cite[Sec. 6]{Gui21}. Since our braiding is the opposite to \cite{Gui21} (see Remark \ref{Rk: Opposite}), it is also the opposite to \cite{MR1410566}. Hence, transporting their result to our setting requires us to also take the opposite balance, that is, its inverse.

Let us write $\DHR^{\text{rev}}(\A)$ for the monoidal opposite of $\DHR(\A)$ with the reverse braiding, and idem for $\Rep^{\text{rev}}(\A)$, where $\Rep(\A)$ uses our conventions from the previous section for the tensor product and the braiding, which recall are the opposite to \cite{Gui21}. The analogue to our conventions of the functor $\DHR(\A)\to \Rep^{\text{rev}}(\A)$ in \cite[Rk. 6.7]{Gui21} produces an equivalence of braided tensor categories $\DHR^{\text{rev}}(\A)\cong \Rep(\A)$.

The equivalence $\DHR^{\text{rev}}(\A)\cong \Rep(\A)$ is an equivalence of braided unitary fusion categories, hence an equivalence of balanced tensor categories if one takes the canonical balance induced by the unitary rigid structure on both sides. Fix $\rho\in \DHR^{\text{rev}}(\A)$ and let $H_\rho\in \Rep(\A)$ be the image of $\rho$ under the functor $\DHR^{\text{rev}}(\A)\to \Rep(\A)$ implementing the equivalence $\DHR^{\text{rev}}(\A)\cong \Rep(\A)$ above. Then, the analogue to \cite[Thm. 3.13]{MR1410566} for our conventions implies that $\theta_{H_\rho}'$ is given by the action $e^{-2\pi i L_0}.$ Since the functor $\DHR^{\text{rev}}(\A)\to \Rep(\A)$ is essentially surjective, we obtain the following~result.

\begin{theorem}\label{Thm: SpinStatistics}
    Let $\A$ be a rational conformal net, and denote by $\theta'$ the canonical balance on the unitary fusion category $\Rep(\A)$. Let $\theta$ be the balance on $\Rep(\A)$ from Theorem \ref{Thm: RepAutAIsCrossedBalanced}. Then, the balances $\theta$ and $\theta'$ agree.
\end{theorem}

\bibliographystyle{halpha-abbrv}
\bibliography{Balanced}
\end{document}